\documentclass {amsart}
\usepackage{amssymb}
\usepackage{stmaryrd}
\usepackage{multirow}

\newtheorem{thm}{Theorem}[section]
\newtheorem{lem}[thm]{Lemma}

\numberwithin{equation}{section}

\newcommand{\ra}{{\longrightarrow}}
\newcommand{\nra}{{\longarrownot\longrightarrow}}
\newcommand{\gen}{\text{gen}}
\newcommand{\cls}{\text{cls}}

\newcommand{\z}{{\mathbb Z}}
\newcommand{\q}{{\mathbb Q}}

\newcommand{\ds}{\displaystyle}

\begin{document}

\title{Almost 2-universal diagonal quinary quadratic forms}


\author{Myeong Jae Kim}
\address{Department of Mathematical Sciences, Seoul National University, Seoul 151-747, Korea}
\email{device89@snu.ac.kr}





\begin{abstract}

A (positive definite integral) quadratic form is called {\it almost 2-universal} if it represents all (positive definite integral) binary quadratic forms except those in only finitely many equivalence classes. Oh \cite{oh} determined all almost 2-universal quinary diagonal quadratic forms remaining three as candidates. In this article, we prove that those three candidates are indeed almost 2-universal.

\end{abstract}

\maketitle

\section{Introduction}

M.-H. Kim and his collaborators proved in \cite{kko} that there are exactly 11 quinary 2-universal quadratic forms. Hwang \cite{hw} proved that there are exactly 3 quinary diagonal quadratic forms that represents all binary quadratic forms except only one. Oh \cite{oh} proved that there exist only finitely many quinary quadratic forms that represent all but at most finitely many equivalence classes of binary quadratic forms. Such quadratic forms are called {\it almost 2-universal} quadratic forms. And he provided a list of almost 2-universal quinary {\it diagonal} quadratic forms, including 3 unconfirmed candidates. In this article, we show that those 3 candidates are indeed almost 2-universal.

\vskip 1cm
\section{Preliminaries and tools}

We adopt lattice theoretic language. Let $\q$ be the rational number field. For a prime (including $\infty$), let $\q_p$ be the fields of $p$-adic completions of $\q$, in particular $\q_\infty = \mathbb{R}$, field of real numbers. For a finite prime $p$, $\z_p$ denotes the $p$-adic integer ring. Let $R$ be the ring of integers $\z$ or the ring of $p$-adic integers $\z_p$. An $R$-{\it lattice} $L$ is a free $R$-module of finite rank equipped with a non-degenerate symmetric bilinear form $B : L \times L \rightarrow R$. The corresponding quadratic map is denoted by $Q$. For a $R$-lattice $L = R \mathbf{e}_1 + R \mathbf{e}_2 + \cdots + R \mathbf{e}_n$ with basis $\mathbf{e}_1, \mathbf{e}_2, \cdots, \mathbf{e}_n$, we write
$$
L = \big( B(\mathbf{e}_i, \mathbf{e}_j ) \big).
$$
For $R$-sublattices $L_1, L_2$ of $L$, we write $L = L_1 \perp L_2$ when $L = L_1 \oplus L_2$ and $B(\mathbf{v}_1, \mathbf{v}_2) = 0$ for all $\mathbf{v}_1 \in L_1, \, \mathbf{v}_2 \in L_2$. If $L$ admits an orthogonal basis $\{ \mathbf{e}_1, \mathbf{e}_2, \cdots, \mathbf{e}_n \}$, we call $L$ {\it diagonal} and simply write
$$
L = \big< Q(\mathbf{e}_1), Q(\mathbf{e}_2), \cdots, Q(\mathbf{e}_n) \big>.
$$
We call $L$ {\it non-diagonal} otherwise. Define the {\it discriminant} $dL$ of $L$ to be the determinant of the matrix $\big( B( \mathbf{e}_i, \mathbf{e}_j ) \big)$. Note that $dL$ is independent of the choice of a basis up to unit squares of $R$. We define {\it scale} $\mathfrak{s}L$ of $L$ to be the ideal of $R$ generated by $B(\mathbf{v}, \mathbf{w})$ for all $\mathbf{v,w} \in L$, norm $\mathfrak{n}L$ of $L$ to be the ideal of $R$ generated by $Q(\mathbf{v})$ for all $\mathbf{v} \in L$. For $a \in R^\times$, we denote by $L^a$ the $R$-lattice obtained from scaling $L$ by $a$.

Let $\ell, L$ be $R$-lattices. We say $L$ {\it represents} $\ell$ if there is an injective linear map from $\ell$ into $L$ that preserves the bilinear form, and write $\ell \ra L$. Such a map will be called a {\it representation}. A representation is called {\it isometry} if it is surjective. We say two $R$-lattices 
$L,K$ are {\it isometric} if there is an isometry between them, and wrtie $L \cong K$.

For a $\z$-lattice $L$ and a prime $p$, we define the $\z_p$-lattice $L_p := \z_p \otimes L$ and call it  the {\it localization} of $L$ at $p$. The set of all $\z$-lattices that are isometric to $L$ is called the {\it class} of $L$, denoted by $\cls(L)$. The set of all $\z$-lattices $K$ such that $L_p \cong K_p$ for all prime spots $p$ (including $\infty$) is called {\it genus} of $L$, denoted by $\gen(L)$. The number of non-isometric classes in $\gen(L)$ is called the {\it class number} of $L$, denoted by $h(L)$. The following two properties are well-known. (see \cite[103, 102:5]{om})
\begin{itemize}
\item [(1)] $h(L)$ is finite for any $\z$-lattice $L$.
\item[(2)] For two lattices $\ell, L$, if $\ell_p \ra L_p$ for all prime $p$ (including $\infty$), then $\ell \ra K$ for some lattice $K \in \gen(L)$.
\end{itemize}

For a $\z$-lattice $L$, we say that $L$ is {\it positive definite} or simply {\it positive} if $Q(\mathbf{v}) > 0$ for any $\mathbf{v} \in L, \, \mathbf{v \neq 0}$. Let $L$ be a positive $\z$-lattice. $L$ is called {\it n-universal} if $L$ represents all $n$-ary positive $\z$-lattices. And $L$ is called {\it almost n-universal} if $L$ represents all $n$-ary positive $\z$-lattices except those in only finitely many equivalence classes.
For a fixed prime $p$, $L$ is called {\it n-universal over} $\z_p$ if its localization $L_p$ represents all $n$-ary $\z_p$-lattices. And $L$ is called {\it locally n-universal} if it is $n$-universal over $\z_p$ for all primes $p$.

We will denote for convenience
$$ [a,b,c] := \begin{pmatrix} a & b \\ b & c \end{pmatrix}. $$
Any unexplained notations and terminologies can be found in \cite{ki} or \cite{om}.

Now we provide a technique for representations of binary $\z$-lattices by certain quinary $\z$-lattices, which is based on the proof of the main theorems in \cite{hw} and \cite{kko}. Let $\ell = [a,b,c]$ be a binary $\z$-lattice. For any integers $n,s,t$, we define
$$ \ell_{s,t}^n := \begin{pmatrix} a-ns^2 & b-nst \\ b-nst & c-nt^2 \end{pmatrix}. $$
Let $M$ be a $\z$-lattice. It can be verified that $\ell \ra M \perp \langle n \rangle$ if and only if there exist integers $s,t$ such that $\ell_{s,t}^n \ra M$. If the class number of the lattice $M$ is one, we can classify all binary lattices which are represented by $M$ using the local representation theory. In this thesis, we only consider the case that $M$ is quaternary. Let $p$ be a prime such that $p \nmid 2 dM$. Then $M_p \simeq \langle 1,1,1,dM \rangle$. If $dM$ is a square in $\z_p$, then $M$ is 2-universal over $\z_p$. If $dM$ is not a square, then $\ell_p \ra M_p$ if and only if $\ell_p$ is not isometric to any sublattices of the $\z_p$-lattice $\langle p, -p \varDelta \rangle$, where $\varDelta$ is a non-square unit in $\z_p$. In particular,  $\frak{s} \ell \not\subseteq p \z$ implies that $\ell_p \ra M_p$.

Let $\mathfrak{P}$ be the set of primes $p$ such that $\ds \bigg( \frac{dM}{p} \bigg) = -1$. We will choose $s,t$ such that $\gcd(a-ns^2, b-nst)$ has no prime factors in $\mathfrak{P}$. Then the scale of $\ell_{s,t} = [a-ns^2,b-nst,c-nt^2]$ is not contained in $p \z$ for any prime $p \in \mathfrak{P}$. Thus we may only consider the $\z_p$-structure for primes $p \mid 2 dM$. 

Consider the case that $n$ has a prime factor $q \in \mathfrak{P}$, which is a hard case. If $\mathfrak{s} \ell \subseteq q \z$, then $\mathfrak{s} \ell_{s,t} \subseteq q \z$ for all $s,t$. For this reason, we have to prove this case separately.

Suppose that $\ell_{s,t}^n \ra M$ over $\z_p$ for all primes $p$. If $\ell_{s,t}$ is positive, then we conclude that $\ell_{s,t}^n \ra M$, and that $\ell \ra M \perp \langle n \rangle$. Following lemma says that $\ell_{s,t}^n$ is positive for sufficiently large $a$.

\vskip 10pt
\begin{lem}\label{pos} Let $\ell = [a,b,c]$ be a Minkowski reduced binary $\z$-lattice, that is, $2|b| \le a \le c$. If $a > \ds\frac{4}{3} n ( s^2 + |st| + t^2 )$, then $\ell_{s,t}^n$ is positive.
\end{lem}
\begin{proof} $d \ell_{s,t}^n = ac - b^2 - ns^2 c + 2nst b - nt^2 a =
\ds \frac{1}{4} ac - b^2 + \frac{3}{4} ac - n (s^2 c - 2st b + t^2 a)$\\
$\ge 0 + \ds\frac{3}{4}ac - n(s^2 c + |st| c + t^2 c)
= \frac{3}{4}c \bigg( a - \frac{4}{3} n (s^2 + |st| + t^2) \bigg) > 0 $.
\end{proof}

\vskip 1cm
\section{Main results}

\begin{thm} There are exactly 14 almost 2-universal quinary diagonal $\z$-lattices. Those lattices are:
\begin{table}[ht]
\begin{tabular}{|l|}
\hline \rule[-2mm]{-2.5mm}{7mm} 
 \quad $\langle 1,1,1,1,1 \rangle, \ \langle 1,1,1,1,2 \rangle, \ \langle 1,1,1,1,3 \rangle, \ \langle 1,1,1,2,2 \rangle, \ \langle 1,1,1,2,3 \rangle$ \\  \hline

\rule[-2mm]{-2.5mm}{7mm} 
 \quad $\langle 1,1,1,1,5 \rangle, \ \langle 1,1,1,2,4 \rangle, \ \langle 1,1,1,2,5 \rangle, \ \langle 1,1,1,2,7 \rangle, \ \langle 1,1,2,2,3 \rangle, \ \langle 1,1,2,2,5 \rangle $ \\  \hline

\rule[-2mm]{-2.5mm}{7mm} 
 \quad $\langle 1,1,1,3,7 \rangle, \ \langle 1,1,2,3,5 \rangle, \ \langle 1,1,2,3,8 \rangle$ \\  \hline

\end{tabular}
\vskip 1mm \caption{Almost 2-universal diagonal quinary $\z$-lattices}
\end{table}

\end{thm}

\vskip -10pt

By \cite{oh}, there are eleven almost 2-universal quinary diagonal $\z$-lattices and there can be at most three more. Among the quinary diagonal $\z$-lattices listed in Table 1, five lattices in the first box are, in fact 2-universal, and six lattices in the second box are almost 2-universal (see \cite{kko}, \cite{hw} and \cite{oh}). Three lattices in the third box are candidates for almost 2-universal quinary diagonal $\z$-lattices provided in \cite{oh}.  Now we prove that those three lattices are indeed almost 2-universal.

\vskip 10pt
\begin{thm}\label{137} The quinary $\z$-lattice $L = \langle 1,1,1,3,7 \rangle$ represents all binary lattices except following 19 binary lattices:

$$ [2,1,3], \, [4,1,4], \, \langle 1,6 \rangle, \, \langle 4,6 \rangle, \, [2,1,7], \, [3,1,7], \, [4,2,7], $$
$$ [6,3,7], \, [7,2,7], \, [7,3,9], \,  [7,1,10], \, [10,5,10], \, [7,2,15], $$
$$\langle 10,15 \rangle, \, \langle 6,16 \rangle, \, [7,2,22], \, [7,1,26], \, [7,3,34], \, [10,5,47]. $$
\end{thm}
\begin{proof}
Consider the quaternary sublattice $M = \langle 1,1,1,3  \rangle$ of $L$, which has class number one. Let $\ell:=[a,b,c]$ a binary lattice such that $0 \le 2b \le a \le c$. And we define the binary lattice
$$ \ell_{s,t} := [a-7s^2, b-7st, c-7t^2] $$
for each integers $s,t$. Note that $\ell \ra M \perp \langle 7 \rangle$ if and only if $\ell_{s,t} \ra M$ for some integers $s,t$.

\vskip 10pt
\noindent \textbf{(Step 1)}
First, for the case that $a<30$, we verify that $\ell=[a,b,c] \ra M \perp \langle 7 \rangle$. As a sample, we only consider the case that $a=10, \, b=5$. Other cases can be verified similarly. We use the fact that
$$ [3,1,c] \ra M \quad \textit{if} \ \ c \equiv 0,1,4,5,6 \pmod 8, $$
$$ \langle 3,c \rangle \ra M \quad \textit{if} \ \ c \equiv 1,2,3,5,6 \pmod 8. $$
For a binary lattice $\ell = [10,5,c]$, we can check the followings:
\begin{itemize}
\item If $c \not\equiv 3,6,7 \pmod 8$, then $\ell_{1,2} \simeq \langle 3, c-55 \rangle \ra M$. ($c>55$)
\item If $c \equiv 6 \pmod 8$, then $\ell_{1,1} \simeq [3,1,c-8] \ra M$. ($c>8$)
\item If $c \equiv 7,11 \pmod{16}$, then $\ell_{1,1} \subseteq [3,1, \frac{1}{4}(c-7) ] \ra M$. ($c>7$)
\item If $c \equiv 15 \pmod{16}$, then $\ell_{1,2} \subseteq \langle 3, \frac{1}{4} (c-55) \rangle \ra M$. ($c>55$)
\item If $c \equiv 3 \pmod{16}$, then $\ell_{1,3} \subseteq \ [3,1,\frac{1}{4}(c-147) ] \ra M$. ($c>147$)
\end{itemize}
For a small $c$ such that $\ell_{1,t}$ is not positive, we can also check it by a direct calculation. In this case, it can be verified that three binary lattices $[10,5,3] \simeq [2,1,3], [10,5,10], [10,5,47]$ are not represented by $M \perp \langle 7 \rangle$. In short, 
$$ [10,5,c] \ra M \perp \langle 7 \rangle \quad \forall c \neq 3, 10, 47. $$
For $a<30$, we can verify $\ell \ra L$ except followings:
$$ \langle 1,6 \rangle, \, [2,1,3], \, \langle 4,6 \rangle, \, [4,1,4], \, \langle 6,16 \rangle,
\, \langle 10,15 \rangle, \, [10,5,10], \, [10,5,47], $$
$$ [7,1,c_1] \ (c_1 = 2,3,10,26),  \quad [7,2,c_2] \ (c_2 = 4,7,15,22),  \quad  [7,3,c_3] \ (c_3 = 6,9,34). $$
By a direct calculation, we can verify that sublattices of above exceptions with index a power of 2 are represented by $L$. Hereafter, we only consider $\z_2$-primtive binary lattices $\ell$.

\vskip 10pt
\noindent \textbf{(Step 2)} For a binary lattice $\ell = [a,b,c]$, suppose that $a,c \ge 30$ and $0 \le 2b \le a \le c$. And suppose that $\mathfrak{s} \ell \not\subseteq 7 \z$. 

By checking the local structure of $\ell_{s,t}$ and $M$ over $\z_2, \z_3$, we obtain the following properties.
\begin{itemize}
\item[(1.1)] If $(a,b,c) \equiv (1,0,1) \pmod 2$ and $(s,t) \equiv (1,1) \pmod 2$, then $\ell_{s,t} \ra M$ over $\z_2$.
\item[(1.2)] If $(a,b,c) \equiv (0,1,0) \pmod 2$ and $(s,t) \equiv (0,0) \pmod 2$, then $\ell_{s,t} \ra M$ over $\z_2$.
\item[(1.3)] If $(a,b,c) \equiv (1,1,0) \pmod 2$ and $(s,t) \equiv (1,0) \pmod 2$, then $\ell_{s,t} \ra M$ over $\z_2$.
\item[(1.4)] If $(a,b,c) \equiv (0,1,1) \pmod 2$ and $(s,t) \equiv (0,1) \pmod 2$, then $\ell_{s,t} \ra M$ over $\z_2$.
\item[(1.5)] If $a \equiv 5 \pmod 8$ or $c \equiv 5 \pmod 8$, then $\ell_{s,t} \ra M$ over $\z_2$ for any $s,t$.
\item[(1.6)] If $a \equiv 1 \pmod 8, \, 2 \mid s$ or $c \equiv 1 \pmod 8, \, 2 \mid t$, then $\ell_{s,t} \ra M$ over $\z_2$.
\item[(1.7)] If $4 \mid a, \, 2 \nmid s$ or $4 \mid c, \, 2 \nmid t$, then $\ell_{s,t} \ra M$ over $\z_2$.
\item[(1.8)] If $a \equiv 3 \pmod 4$ or $c \equiv 3 \pmod 4$, $2 \mid b$ and $2 \nmid st$, then $\ell_{s,t} \ra M$ over $\z_2$.
\item[(2.1)] If $3 \mid ac$ and $3 \nmid st$, then $\ell_{s,t} \ra M$ over $\z_3$.
\item[(2.2)] If $(a,b,c) \equiv (1,0,1),(1,0,2),(2,0,1) \pmod 3$ and $3 \nmid st$, then $\ell_{s,t} \ra M$ over $\z_3$.
\item[(2.3)] If $(a,b,c) \equiv (1,2,1),(1,2,2),(2,1,2),(2,2,1) \pmod 3$ and $st \equiv 1 \pmod 3$, then $\ell_{s,t} \ra M$ over $\z_3$.
\item[(2.4)] If $(a,b,c) \equiv (1,1,1),(1,1,2),(2,1,1),(2,2,2) \pmod 3$ and $st \equiv 2 \pmod 3$, then $\ell_{s,t} \ra M$ over $\z_3$.
\item[(2.5)] If $(a,b,c) \equiv (2,0,2) \pmod 3$ and $st \equiv 0 \pmod 3$, then $\ell_{s,t} \ra M$ over $\z_3$.
\end{itemize}

Under the assumption that the binary lattice $\ell$ is $\z_2$-primitive, above conditions cover all cases. For example, the case that $a \equiv c \equiv 3 \pmod 4$ and $2 \nmid b$ is not contained in above. However in this case $\ell$ is not $\z_2$-primitive. $\z_3$-primitivity of $\ell$ is not necessary. Regardless of $\z_3$-primitivity of $\ell$, all cases are contained in above.

For all cases, we may choose $s=1,2$ and $t \equiv i \pmod 6$ for some $i$. Each case can be proved similarly. We only consider the case that $\ell=[a,b,c]$ satisfies the conditions given in both (1.4) and (2.3). In this case, $\ell_{s,t} \ra M$ over $\z_2, \z_3$ if $s=2$ and $t \equiv 1 \pmod 6$.

Let $\mathfrak{P} = \{5, 7, 17, 19, 29, 31, ... \}$ be the set of primes $p$ such that $\ds \bigg( \frac{3}{p} \bigg) = -1$. From the assumption that $\mathfrak{s} \ell \not\subseteq 7 \z$, we get $\mathfrak{s} \ell_{s,t} \not\subseteq 7 \z$ for all $s,t$, and hence $\ell_{s,t} \ra M$ over $\z_7$. Let $p_1, p_2, ..., p_k$ be the primes in $\mathfrak{P} - \{ 7 \}$ dividing $a-28$. We want choose a suitable $t$ such that $b-14t$ is relatively prime to $p_1 p_2 ... p_k$. Then we get $\ell_{s,t} \ra M$ over $\z_p$ for all $p \neq 2,3$.

If $k=0$, then $\ell_{2,1} \ra M$. By lemma \ref{pos}, $\ell_{2,1}$ is positive if $a \ge 66$. In the case that $30 \le a \le 65$, one can show that $\ell_{2,1}$ is also positive for sufficiently large $c$. In the case that $a=34$, for example, $\ell_{2,1}$ is positive whenever $c>39$. The remaining cases are finitely many and we can check it by a direct calculation.

If $1 \le k \le 4$, then $\ell_{2,t} \ra M$ for some $t \in \{ 6m+1 \mid \ds - \Big[ \frac{k+1}{2} \Big] \le m \le \Big[ \frac{k}{2} \Big] \}$. Note that $a \ge 28 + p_1 \cdots p_k$. In the case that $k=3,4,5$, all $\ell_{2,t}$ are positive by lemma \ref{pos}. In case that $k=1,2$, however, the positiveness of $\ell_{2,t}$ is not guaranteed. There are only finitely many cases such that $\ell_{2,t}$ is not positive. When $\ell_{2,t}$ is not positive, we can check that $\ell \ra L$ by a direct calculation.

If $k = 5$, then $\ell_{2,t} \ra M$ for some $t \in \{ -17, -11, ... 13, 19 \}$. Since $a \ge 28 + 5 \cdot 17 \cdot 19 \cdot 29 \cdot 31$, $\ell_{2,19}$ is positive.

If $k \ge 6$, then $\ell_{2,t} \ra M$ for some $t \in \{ -3k 2^k +1 , ... -5,1, ..., 3k 2^k - 5 \}$. Since $a \ge 28 + 5 \cdot 17 \cdot 19 \cdot 29 \cdot 31 \cdot 41 \cdot 43^{k-6}$, all $\ell_{2,t}$ are positive by lemma 3 of \cite{kko}.

\vskip 10pt
\noindent \textbf{(Step 3)} We show that $\ell \ra L$ when $\mathfrak{s} \ell \subseteq 7 \ $, 
that is, $\ell$ is the form of $[7a,7b,7c]$. Let $\ell' = [a,b,c]$, then $\ell = \ell'^7$. Consider the quaternary lattice $K = \langle 1 \rangle \perp \begin{pmatrix} 2 & 1 & 0 \\ 1 & 2 & 1 \\ 0 & 1 & 3 \end{pmatrix}$. Note that $K$ has class number one and
$$ \big( K \perp \langle 21 \rangle \big)^7 \ra L.$$
If $\ell' \ra K \perp \langle 21 \rangle$, then $\ell'^7 = \ell \ra L$. We only consider the case that $\ell'^7$ is $\z_7$-primitive. Thus we may assume that $\ell' \simeq \langle 1, - \varDelta \rangle$ over $\z_7$, where $\varDelta$ is a nonsquare unit in $\z_7$. This is equivalent to
 $$d \ell' \equiv 1,2,4 \pmod 7.$$
Define $ \ell'_{s,t} = [a-21s^2, b-21st, c-21t^2]$. From the fact that $7 \nmid d \ell'$ we get $7 \nmid d \ell'_{s,t}$ and $\ell'_{s,t} \ra K$ over $\z_7$ for any $s,t$. The followings are the sufficient conditions such that $\ell'_{s,t} \ra K$ over $\z_2$ assuming that $\ell'$ is $\z_2$-primitive.
\begin{itemize}
\item[(1)] $(a,b,c) \equiv (1,0,1) \pmod 2$ and $(s,t) \equiv (1,1) \pmod 2$.
\item[(2)] $(a,b,c) \equiv (0,0,1),(0,1,0),(1,0,0),(1,1,1) \pmod 2$ and $(s,t) \equiv (0,0) \pmod 2$.
\item[(3)] $(a,b,c) \equiv (1,1,0) \pmod 2$ and $(s,t) \equiv (1,0) \pmod 2$.
\item[(4)] $(a,b,c) \equiv (0,1,1) \pmod 2$ and $(s,t) \equiv (0,1) \pmod 2$.
\end{itemize}
Using the same method as step 1,2, we get $\ell' \ra K \perp \langle 21 \rangle$ with finitely many exceptions. For example, $\langle 5,61 \rangle, \, [9,2,13] \nra K \perp \langle 21 \rangle$. By brute force computation, one can show that the binary lattices obtained by scaling these excepstions are indeed represented by $L$. Therefore we conclude that $\ell \ra L$ for all binary lattices whose scale is contained in $7 \z$.
\end{proof}

\vskip 10pt
For the other two lattices, the proofs are quite similar to the above. We only provide following data for the proof of almost 2-universality of $L$:
\begin{itemize}
\item[(1)] Quternary sublattice $M$ which has class number one
\item[(2)] The integer $n$ satisfying $M \perp \langle n \rangle \ra L$
\item[(3)] Conditions such that $\ell_{s,t}^n \ra M$ over $\z_p$ where $p \mid 2 dM$
\item[(4)] Data for the case that $\mathfrak{s} \ell \subseteq q \z$ where $q \mid n$ and $\ds \bigg( \frac{dM}{q} \bigg) = -1$
\end{itemize}

\vskip 10pt
\begin{thm}\label{238} {\rm (1)} The quinary $\z$-lattice $\langle 1,1,2,3,5 \rangle$ represents all binary lattices except $$ [2,1,2], \, [5,2,5], \, [6,3,6].$$
{\rm (2)} The quinary $\z$-lattice $\langle 1,1,2,3,8 \rangle$ represents all binary lattices except the following 15 binary lattices:
$$ [2,1,2], \langle 2,6 \rangle, \, \langle 5,6 \rangle, \, [5,2,5], \, [5,1,8], \, [6,3,6], \, [6,3,8], \,  [6,3,14], $$
$$ \langle 6,26 \rangle, \, \langle 7,7 \rangle, \, [8,4,8], \, \langle 10,33 \rangle, \, \langle 11,14 \rangle, \, \langle 14,14 \rangle, \, [25,3,25]. $$
\end{thm} 
\begin{proof}
Set $M = \langle 1,1,2,3 \rangle, \, n=5,8$

\vskip 10pt \noindent
Conditions such that $\ell_{s,t}^n \ra M$ over $\z_3$ where $n=5,8$: ($\z_3$-primitivity of $\ell$ is not necessary)
\begin{itemize}
\item $3 \mid ac$ and $3 \nmid st$
\item $(a,b,c) \equiv (1,0,2),(2,0,1),(2,0,2) \pmod 3$ and $3 \nmid st$
\item $(a,b,c) \equiv (1,1,2),(1,2,1),(2,1,1),(2,1,2) \pmod 3$ and $st \equiv 1 \pmod 3$
\item $(a,b,c) \equiv (1,1,1),(1,2,2),(2,2,1),(2,2,2) \pmod 3$ and $st \equiv 2 \pmod 3$
\item $(a,b,c) \equiv (1,0,1) \pmod 3$ and $st \equiv 0 \pmod 3$
\end{itemize}

\vskip 10pt \noindent
Conditions such that $\ell_{s,t}^5 \ra M$ over $\z_2$: ($\ell$ is $\z_2$-primitive)
\begin{itemize}
\item $(a,b,c) \equiv (0,1,0) \pmod 2$ and $\forall s,t$
\item $(a,b,c) \equiv (0,0,1),(1,0,0),(1,0,1) \pmod 2$ and $(s,t) \equiv (1,1) \pmod 2$
\item $(a,b,c) \equiv (0,1,1),(1,1,1) \pmod 2$ and $(s,t) \equiv (0,1) \pmod 2$
\item $(a,b,c) \equiv (1,1,0),(1,1,1) \pmod 2$ and $(s,t) \equiv (1,0) \pmod 2$
\end{itemize}

\vskip 10pt \noindent
Conditions such that $\ell_{s,t}^8 \ra M$ over $\z_2$: ($\ell$ is $\z_2$-primitive)
\begin{itemize}
\item $(a,b,c) \equiv (0,1,0),(0,1,1),(1,1,0),(1,0,1) \pmod 2$ and $\forall s,t$
\item $a \equiv 5,7 \pmod 8$ or $c \equiv 5,7 \pmod 8$ and $\forall s,t$
\item $a \equiv 10,14 \pmod{16}, \, 2 \nmid s$ or $c \equiv 10,14 \pmod{16}, \, 2 \nmid t$
\item $(a,c) \equiv (1,3), (9,11) \pmod{16}, \, b \equiv \pm 1 \pmod 8$ and $(s,t) \equiv (1,1) \pmod 2$
\item $(a,c) \equiv (1,11), (3,9) \pmod{16}, \, b \equiv \pm 3 \pmod 8$ and $(s,t) \equiv (1,1) \pmod 2$
\item $(a,c) \equiv (1,3), (9,11) \pmod{16}, \, b \equiv \pm 3 \pmod 8$ and $(s,t) \equiv (0,1),(1,0) \pmod 2$
\item $(a,c) \equiv (1,11), (3,9) \pmod{16}, \, b \equiv \pm 1 \pmod 8$ and $(s,t) \equiv (0,1),(1,0) \pmod 2$
\item $(a,c) \equiv (1,6),(9,6),(3,2),(11,2) \pmod{16}, \, b \equiv 2 \pmod 4$ and $(s,t) \equiv (1,0) \pmod 2$
\item $(a,c) \equiv (1,2),(9,2),(3,6),(11,6) \pmod{16}, \, b \equiv 0 \pmod 4$ and $(s,t) \equiv (1,0) \pmod 2$
\item $(a,c) \equiv (1,6),(9,6),(3,2),(11,2) \pmod{16}, \, b \equiv 0 \pmod 4$ and $(s,t) \equiv (1,1) \pmod 2$
\item $(a,c) \equiv (1,2),(9,2),(3,6),(11,6) \pmod{16}, \, b \equiv 2 \pmod 4$ and $(s,t) \equiv (1,1) \pmod 2$
\end{itemize}
Since there are no prime factors $q$ such that $q \mid n$ and $\ds \bigg( \frac{dM}{q} \bigg) = -1$, in this case, the process such as step 3 in the proof of Theorem \ref{137} is not necessary.
\end{proof}


\begin{thebibliography}{abcd}

\bibitem {ha} P.R. Halmos, {\em Note on almost-universal forms}, Bull. Amer. Math. Soc. \textbf{44}(1938), 141-144

\bibitem {hj} J.S. Hsia, M. J$\rm\ddot{o}$chner, {\em Almost strong approximations for definite quadratic spaces}, Invent. Math. \textbf{129}(1997), 471-487

\bibitem {hw} D.-S. Hwang, {\em On almost $2$-universal integral quinary quadratic forms}, Ph. D. Thesis, Seoul National Univ. (1997). 

\bibitem {kko}  B.M. Kim, M.-H. Kim and B.-K. Oh, {\em $2$-universal positive definite integral quinary quadratic forms}, Contem. Math. \textbf{249}(1999), 51--62.

\bibitem {ki} Y. Kitaoka, {\em Arithmetic of quadratic forms}, Cambridge University Press, Cambridge (1993).

\bibitem {kl} H. D. Kloosterman, {\em On the representation of numbers in the form $ax^2 + by^2 + cz^2 + dt^2$}, Acta Math. \textbf{49}(1926), 407-464

\bibitem {oh} B.-K. Oh, {\em The representation of quadratic forms by almost universal forms of higher rank}, Math. Z. \textbf{244}(2003), 399--413.

\bibitem {om} O. T. O'Meara, {\em Introduction to quadratic forms}, Springer, New York (1963).

\bibitem {pall} G. Pall, {\em The completion of a problem of Kloosterman}, Amer. J. Math. \textbf{68}(1946), 47-58

\bibitem {pall_ross} G. Pall, A. Ross, {\em An extension of a problem of Kloosterman}, Amer. J. Math. \textbf{68}(1946), 59-65

\bibitem {ra} S. Ramanujan, {\em On the expression of a number in the form $ax^2 + by^2 + cz^2 + dw^2$}, Proc. Cambridge Phil. Soc. \textbf{19}(1917), 11--21. 

\bibitem {w} M. F. Willerding, {\em Determination of all classes of (positive) quaternary quadratic forms which represent all positive integers}, Bull. Amer. Math. Soc. \textbf{54}(1948), 334--337.


\end{thebibliography}
\end{document}